\documentclass{article}
\usepackage{graphicx} 
\usepackage{amsmath}
\usepackage{amsthm}
\usepackage{amssymb}
\usepackage{xcolor}
\usepackage[hidelinks]{hyperref}
\newtheorem{theorem}{Theorem}
\newtheorem{lemma}[theorem]{Lemma}

\theoremstyle{definition}
\newtheorem{definition}{Definition}[section]
\title{On the distribution of the minimal length of addition chains}
\author{Jean-Marie De Koninck, Nicolas Doyon, William Verreault}
\date{}

\begin{document}

\maketitle

\begin{abstract}

A sequence of integers $1=a_0<a_1<\cdots<a_k=n$ is called an addition chain of length $k$ if $a_j=a_s+a_t$ with $0\le s,t<j$ for all integers $j\in \{1,2,\ldots,k\}$. We denote by $\ell(n)$ the minimal length of an addition chain leading to $n$. Here we investigate the distribution of the function $\ell$ through the counting function
$$
F(m,r):=\#\{n\in [2^m,2^{m+1}):\ell(n)\le m+r\}
$$
and show that, for every fixed $0<c<\log 2$, there exist
positive constants $K_1$ and $K_2$
such that
$$
K_1^r m^r\le F(m,r)\le K_2^r m^r
$$
for all sufficiently large $m$ and all integers
$m^{0.9}<r\le cm/\log m$. The upper bound also holds for every integer $r>m^{0.9}$. Moreover, denoting by $G(m,r)$ the number of \emph{distinct} addition chains of length $m+r$ leading to an integer $n\in [2^m, 2^{m+1})$, we show that there exist positive constants $K_3$ and $K_4$ such that
$$
K_3^r \left(\frac{m^2}{r}\right)^r\le G\left(m,r\right)\le  K_4^r \left(\frac{m^2}{r} \right)^r
$$
provided $m^{0.9}<r<m$. This improves and generalizes previous results on the minimal length of addition chains and addresses a question raised by Paul Erd\H{o}s.
\end{abstract}

\section{Introduction}
We say that a sequence of integers $1=a_0<a_1<\cdots<a_k=n$ is an {\it addition chain} of length $k$ if $a_j=a_s+a_t$ with $0\le s,t<j$ for all $j,\, 1\le j\le k$. In other words, an element of an addition chain must be obtained by adding two previous ones.

For example, the following is an addition chain:
\begin{eqnarray*}
a_0&=&1,\\
a_1&=&a_0+a_0=1+1=2,\\
a_2&=&a_1+a_1=2+2=4,\\
a_3&=&a_2+a_0=4+1=5,\\
a_4&=&a_3+a_2=5+4=9,\\
a_5&=&a_4+a_0=9+1=10.
\end{eqnarray*}

We denote by $\ell(n)$ the shortest possible length of an addition chain leading to $n$.  For instance, if $n=10$, the following addition chain leads to $10$: $a_0=1$, $a_1=1+1=2$, $a_2=2+2=4$, $a_3=4+4=8$, $a_4=8+2=10$.  One can easily show that there exists no shorter addition chain leading to $10$, thereby proving that $\ell(10)=4$.

The minimal length of addition chains has been investigated by several authors, in particular (in chronological order) by Scholz \cite{scholz}, Brauer \cite{brauer}, Erd\H{o}s \cite{erdos}, Thurber \cite{thurber2}, Sch\"onhage \cite{schonhage}, Thurber \cite{thurber1}, Bahig \& Nakamula \cite{bahig}, Clift \cite{clift}, Tall \cite{tall} and De Koninck, Doyon \& Verreault \cite{DKN}. A partial survey of the study of addition chains is presented in section C6 of the book of Richard Guy \cite{guy}. Generalizations of the concept of addition chains have also been investigated by a few authors such as Elias and McKenzie \cite{Elias}, and Järvinen, Dimitrov and Azarderakhsh \cite{Jarvinen}, who considered the case where several of the previous integers can be added at each step. Addition chains also arise naturally in the efficient computation of powers, since each step corresponds to a multiplication. For example,
continued fraction methods for constructing short addition chains were developed by Bergeron, Berstel and Brlek \cite{BBB}. Related algorithmic questions continue to be investigated for constrained families of chains such as differential addition chains arising in elliptic curve computations \cite{BCL}.

The investigation of the function $\ell(n)$ began with Arnold Scholz, who in 1937 formulated \cite{scholz} what is now known as the {\it Scholz conjecture} (and at times as the {\it Scholz--Brauer conjecture}), which states that
$$\ell(2^{n+1}-1)\le n +\ell(n+1).$$

In 1939, Alfred Brauer \cite{brauer} obtained the upper bound
$$\ell(n) \le \min_{1\le r \le n} \left\{ \left( 1+ \frac 1r \right) \log_2 n +2^r -2 \right\}$$
where $\log_2 n$ denotes the base 2 logarithm of $n$.
Choosing $\displaystyle{ r =\left\lfloor (1-\varepsilon) \frac{\log \log n}{\log 2}  \right\rfloor}$, and then relabelling $\varepsilon$, Erd\H{o}s deduced \cite{erdos} that for $n$ sufficiently large
\begin{equation*}
\ell(n) \le \log_2 n +(1+\varepsilon) \frac{\log n}{\log \log n}
\end{equation*}
for any real $\varepsilon>0$.

Bounds that depend on the number of digits equal to 1 in the binary expansion of $n$ (here denoted by $\nu(n)$) have also been established.
In 1975, Arnold Sch\"onhage \cite{schonhage} proved that
$$\ell(n)\ge \log_2 n + \log_2 \nu(n) -2.13.$$

An important result on the typical length of addition chains was obtained by Paul Erd\H{o}s in 1960 as he showed \cite{erdos} that for almost all $n$ (that is, for all $n$ except perhaps on a set of density 0),
$$\ell(n) = \log_2 n + (1+o(1)) \frac{\log n}{\log \log n} \qquad (n\to \infty).$$
Erd\H{o}s also raised the question of obtaining more precise estimates for $\ell(n)$ and, in particular, an asymptotic distribution function. In this work, we provide quantitative information on the lower tail of this distribution. 

Here, we further investigate the distribution of the function $\ell(n)$ through the counting function
$$
F(m,r):=\#\left\{n\in [2^m, 2^{m+1}):\ell(n)\leq m+r\right\},
$$
where $m$ and $r$ are positive integers. Since $\ell(n)\ge \log_2 n$ and $m\le \log_2 n<m+1$ on this interval, it is natural to study the excess $\ell(n)-m$. The normalized count $2^{-m}F(m,r)$ is precisely the distribution function of this excess when $n$ is chosen uniformly among the integers in $[2^m,2^{m+1})$. Erd\H{o}s' theorem shows that the
excess is typically asymptotic to $m\log 2/\log m$.
Thus, the range $r\le cm/\log m$ with $c<\log 2$ describes integers admitting addition chains shorter than the typical minimal length.
The use of a dyadic interval $[2^m,2^{m+1})$ is also motivated by the fact that {\it doubling steps}, that is, steps for which
$$
a_{j+1}=2a_j,
$$
play a critical role in the study of addition chains.
Investigating the function $F(m,r)$ in the critical domain $r=\lfloor cm/\log m \rfloor$, the authors previously showed \cite{DKN} that
$$
F\left(m,\left\lfloor \frac{cm}{\log m}\right\rfloor\right)\le \exp\left(cm+o\left(\frac{m\log\log m}{\log m}\right)\right)
$$
for any positive constant $c<\log 2$ as $m\to\infty$.
In the same paper, they showed that for any $\varepsilon>0$,  we have for $m$ large enough
$$
F\left(m,\left\lfloor \frac{cm}{\log m}\right\rfloor\right)\ge \exp\left(cm-\frac{c(1+\varepsilon) m\log\log m}{\log m}\right).
$$
Here, we refine and generalize these bounds to  obtain matching upper and lower estimates.

For a given integer $n$, there can be several distinct addition chains of minimal length leading to $n$.  For example, if $n=5$, we have that $(1,2,4,5)$ and $(1,2,3,5)$  are two addition chains of minimal length, and it is easy to see that there are no others. To capture this, we define the function $G(m,r)$ as the number of distinct addition chains of length  $ m+r$, where $r$ is a positive integer, leading to an integer $n\in [2^m, 2^{m+1})$. 
Even though it has been less studied, the function $G(m,r)$   is of much interest since it appears naturally when investigating the function $F(m,r)$.

In our main two theorems, we provide lower and upper bounds  for the functions $G$ and $F$.

\begin{theorem} \label{t1} Assuming that $m^{0.9}<r<m$ with $m$ and $r$ sufficiently large, there exist  computable constants $K_3>0$  and $K_4>0$ such that
$$
K_3^r \left(\frac{m^2}{r}\right)^r\le G(m,r)\le K_4^r \left(\frac{m^2}{r}\right)^r.
$$
\end{theorem}

\begin{theorem}\label{t2}
For every fixed $0<c<\log 2$, there exist computable constants $K_1>0$  and $K_2>0$ such that
$$
 K_1^r m^r\le F(m, r)\le K_2^r m^r
$$
for all sufficiently large $m$ and all integers
$m^{0.9}<r\le cm/\log m$.
The upper bound also holds for every integer $r>m^{0.9}$.
\end{theorem}

Note that the lower bound in Theorem \ref{t2} cannot hold for every $m^{0.9}<r<m$. Since $F(m,r)\leq 2^m$, taking $r=\lfloor m/2\rfloor$ would contradict any bound $F(m,r)\geq K_1^rm^r$ with fixed $K_1>0$.

Choosing $r=\lfloor cm/\log m\rfloor $ in Theorem \ref{t2}, we obtain
$$
\log F\left(m, r \right) =cm + O_c\left(\frac m{\log m}\right),
$$
thereby improving results established in our previous work \cite{DKN}.

Throughout the paper, $\log$ denotes the natural logarithm, whereas
$\log_2$ denotes the logarithm to base $2$.
Unless otherwise indicated, all constants, as well as all constants implicit
in $O(\cdot)$ and $\asymp$, are absolute.  A subscript as in
$O_c(\cdot)$ or $\asymp_c$ indicates that the corresponding constants may
depend on $c$.




\section{Proving the upper bound in Theorem \ref{t1}}

The proof of the upper bound in Theorem \ref{t1} relies on a series of preliminary definitions and lemmas. We define $\gamma=\frac{1+\sqrt{5}}{2}$
as the golden ratio.

\begin{definition} Consider the addition chain
$$
1=a_0<a_1<\cdots <a_{m+r}=n
$$
where $r$ is a positive integer and assume that $2^m\le n< 2^{m+1}$.  We divide the steps of such an addition chain into the following four sets:
\begin{itemize}
\item ${\cal A}:=\{j,\, 1\le j\le m+r:a_j=2a_{j-1}\}$, the doubling steps,
\item ${\cal B}:=\{j,\, 1\le j\le m+r:\gamma a_{j-1}\leq a_j<2a_{j-1}\}$, the large steps,
\item ${\cal C}:=\{j,\, 1\le j\le m+r:(1+\delta)a_{j-1}\leq a_j<\gamma a_{j-1}\}$, the medium-sized steps,
\item ${\cal D}:=\{j,\, 1\le j\le m+r:a_j<(1+\delta)a_{j-1}\}$, the small steps.
\end{itemize}
Here $\delta=\delta(m)$ is a small real number tending to zero as $m$ tends to infinity that we will specify later. We denote by $A$, $B$, $C$ and $D$ the respective cardinalities of these sets so that $A+B+C+D=m+r$.
\end{definition}
We consider an addition chain as the result of a series  of choices. We begin by fixing the cardinalities $A$, $B$, $C$ and $D$, and from there we   choose the sets ${\cal A}$, ${\cal B}$, ${\cal C}$ and ${\cal D}$. Finally, we choose the integers added in each step belonging to ${\cal B}\cup {\cal C}\cup {\cal D}$. Our strategy is to obtain an upper bound for the number of ways of making each of these choices and to deduce from these an upper bound for the function $G(m,r)$. In the next section, we will extend these ideas to prove an upper bound for $F(m,r)$.

We begin with an elementary observation.
\begin{lemma}\label{ub1}
The number of ways of choosing the cardinalities $A$, $B$, $C$ and $D$ is smaller than or equal to
$$
(m+r)^3.
$$
\begin{proof}
The value of $D$ is completely determined by the choices of $A$, $B$ and $C$ as $D=m+r-A-B-C$.  The number of ways of choosing $A$, $B$ or $C$ is at most $m+r$, so in total it is bounded above by $(m+r)^3$.
\end{proof}
\end{lemma}

The following result provides an upper bound for the number of nondoubling steps.
\begin{lemma}\label{ub2}
We have
$$
B+C+D\le \frac{r}{1-\log_2 \gamma}.
$$
\end{lemma}

\begin{proof}
This is Lemma 4.1 in \cite{DKN}. The proof follows essentially the ideas of Sch\"onhage \cite{schonhage}.
\end{proof}
The choice of the golden ratio $\gamma$ in the definition of the sets ${\cal B}$ and ${\cal C}$ is further motivated by the following observation.

\begin{lemma}\label{ub3}
If a step $j$ belongs to the set ${\cal B}$, then
$$
j-1\in {\cal C}\cup {\cal D}.
$$
\end{lemma}
\begin{proof}
This is Lemma 4.3 in \cite{DKN}.
\end{proof}

From this, we deduce an upper bound for the number of ways one can choose the sets ${\cal A}$, ${\cal B}$, ${\cal C}$ and ${\cal D}$.

\begin{lemma}\label{ub4}
Assuming that the cardinalities $A$, $B$, $C$ and $D$ have been chosen, the number of ways of choosing the sets ${\cal A}$, ${\cal B}$, ${\cal C}$ and ${\cal D}$ is smaller than or equal to
$$
\left(\frac{e(m+r)}{C+D}\right)^{C+D}4^{r/(1-\log_2\gamma)}.
$$
\end{lemma}

\begin{proof}
Observe that the set ${\cal A}$ is completely specified once the sets ${\cal B}$, ${\cal C}$ and ${\cal D}$ are chosen.  From Lemma \ref{ub3}, the number of ways of choosing the set ${\cal B}$ once the sets ${\cal C}$ and ${\cal D}$ have been chosen does not exceed
$$
{C+D\choose B}.
$$
It follows that the number of ways of choosing the sets ${\cal A}$, ${\cal B}$, ${\cal C}$  and ${\cal D}$ is smaller than or equal to
$$
{m+r \choose C+D}{C+D\choose C}{C+D\choose B}\le {m+r \choose C+D} 4^{C+D}.
$$
Using Lemma \ref{ub2}, we deduce that this number of possible choices is smaller than or equal to
$$
{m+r \choose C+D}4^{r/(1-\log_2\gamma)}.
$$
Since we also have that
$$
{m+r\choose C+D}\le \frac{(m+r)^{C+D}}{(C+D)!}\le \left(\frac{e(m+r)}{C+D}\right)^{C+D},
$$
the proof is complete.
\end{proof}
We now turn our attention to bounding the number of ways of choosing the integers added in each step.  If a step belongs to the set ${\cal A}$ (that is, if a step is a doubling step), then there is no choice to be made.

We first obtain an upper bound for the number of ways of choosing integers added in steps belonging to the set ${\cal B}$.
\begin{lemma}\label{ub5}
The number of ways of choosing the integers added in steps belonging to ${\cal B}$ is smaller than or equal to
$$
e^{2r/(e(1-\log_2 \gamma))}.
$$
\end{lemma}

\begin{proof}
Suppose that $j\in {\cal B}$ and that $a_j=a_s+a_t$ with $j-1\ge s\ge t$. First note that $a_t\ge a_s/\gamma$ since otherwise
$$
a_j=a_s+a_t<(1+1/\gamma)a_s=\gamma a_s\le \gamma a_{j-1}.
$$
If $t<s$, this implies that
$$
[t+1,s]\cap({\cal A}\cup {\cal B})=\emptyset.
$$
Now, assume that $s<j-1$ and that there exists an integer $u\in (s, j-1]$ such that $u\in {\cal A}\cup{\cal B}$. Then we would have
$$
a_j> \gamma a_{j-1}\ge \gamma a_u>\gamma^2 a_{u-1}\ge \gamma^2 a_s,
$$
contradicting the inequality $a_s\ge a_j/2$.
From this, we deduce that
$$
[t+1,j-1]\cap({\cal A}\cup {\cal B})=\emptyset.
$$
We thus have that the number of ways of choosing integers added in steps belonging to ${\cal B}$ is smaller than or equal to
$$
\prod_{j\in{\cal B}}(\eta(j)+1)^2,
$$
where $\eta(j)$ is the maximal value of $\eta$ for which
$$
x\in {\cal C}\cup {\cal D},\, \mbox{ for all } x\in[j-\eta,j-1].
$$
In other words, if $j\in{\cal B}$, $\eta(j)$ is the length of the sequence of consecutive integers belonging to ${\cal C}\cup {\cal D}$ and preceding $j$. Since each element of ${\cal C}\cup {\cal D}$ contributes to the value of $\eta(j)$ for at most one $j$, we have
$$
\sum_{j\in{\cal B}}(\eta(j)+1)\le B+C+D.
$$
From this, using the arithmetic-geometric mean inequality, we can conclude that the number of ways of specifying the pairs of indices used in the steps belonging to ${\cal B}$ is smaller than or equal to
\begin{equation*}
\left(\frac{B+C+D}{B}\right)^{2B}\leq \left(\frac{r}{(1-\log_2\gamma)B}\right)^{2B},
\end{equation*}
where we used Lemma \ref{ub2}.

To obtain an upper bound for this last quantity, we consider it as a continuous function of $B$, relaxing the condition that $B$ has to be an integer. Its derivative is equal to
\begin{align*}
&\frac{d}{dB}
\exp\big(2B\log r-2B\log(1-\log_2\gamma)-2B\log B\big)\\
&\quad=
\big(2\log r-2\log(1-\log_2\gamma)-2\log B-2\big)\\
&\qquad\qquad\times
\exp\big(2B\log r-2B\log(1-\log_2\gamma)-2B\log B\big).
\end{align*}
The derivative is positive before
$
B=r/(e(1-\log_2\gamma))
$
and negative afterwards. The claim follows.
\end{proof}

We will now establish an upper bound for the number of ways of choosing the integers added in steps belonging to ${\cal C}$.

\begin{lemma}\label{ub8}
The number of ways of  choosing the integers added in steps belonging to ${\cal C}$ is smaller than or equal to
$$
\left(\frac{-8e^2\log 2(m+r)^2\log \delta}{(\delta C)^2}\right)^{C}.
$$
\end{lemma}
Here, the minus sign may seem problematic. However, since $\delta$ is a small positive real number, $-\log \delta$ is guaranteed to be positive.
\begin{proof}[Proof of Lemma \ref{ub8}]
Suppose that $j\in {\cal C}$ and that $a_j=a_s+a_t$ with $s\ge t$, say.  We then have $a_s\ge a_j/2$ and $$a_t=a_j-a_s\geq a_j-a_{j-1}\geq \frac{\delta}{1+\delta}a_j.$$ 

Define the integer-valued functions $\xi_1(j)$ and $\xi_2(j)$ in the following manner. We set $\xi_1(j)$ as the maximal value of $\xi$ such that
$$
a_{j-\xi}\ge a_j/2,
$$
so that for a fixed $j\in {\cal C}$, the number of ways of choosing $s$ is at most $\xi_1(j)$.  We define $\xi_2 (j)$ as the largest value of $\xi$ such that
$$
a_{j-\xi}\ge \frac{\delta}{1+\delta}a_j,
$$
so that for a fixed $j\in {\cal C}$ the number of ways of choosing $t$ is at most $\xi_2(j)$. 
For every such choice of indices, we have
$$
1\le j-s\le\xi_1(j),
\qquad
1\le j-t\le\xi_2(j).
$$

Let $c_1=c_1(j)$ be the number of elements of ${\cal C}$ in the  interval $(j-\xi_1(j), j]$.  From the definition of ${\cal C}$, we have $(1+\delta)^{c_1(j)}\le 2$, which implies
$$
c_1(j)\log(1+\delta)\le \log 2.
$$
Given that for $\delta$ sufficiently small, we have $\log(1+\delta)>\delta/2$, we obtain that
$$
c_1(j)\le \frac{2\log 2}{\delta}.
$$
Each integer belongs to at most $2\log 2/\delta$ of the
intervals $(j-\xi_1(j),j]$. Indeed, among the intervals
containing that integer, choose the one with the largest
right endpoint. All the right endpoints of these intervals
belong to the chosen interval and to ${\cal C}$, so their
number is bounded by the preceding estimate.
Counting the integer points in these intervals gives
$$
\sum_{j\in {\cal C}} \xi_1(j) \le \frac{2\log 2(m+r)}{\delta}.
$$
Since $j-s\le\xi_1(j)$, the $C$ positive integers $j-s$,
for $j\in{\cal C}$, have total sum at most
$2\log 2(m+r)/\delta$.
Their partial sums form a strictly increasing list of
$C$ integers. Because the positions $j$ are fixed,
the distances $j-s$ determine all the chosen indices $s$.
Thus the number of possible lists of indices $s$ is at most
\begin{equation}\label{new1}
\binom{\lfloor 2(\log 2)(m+r)/\delta\rfloor}{C}
\le
\left(\frac{2e(\log 2)(m+r)}{\delta C}\right)^C.
\end{equation}

Let $c_2=c_2(j)$ be the number of elements of ${\cal C}$ belonging to the interval ${(j-\xi_2(j),j]}$. We have $(1+\delta)^{c_2}\le (1+\delta)/\delta$, from which we deduce that
$$
c_2\log(1+\delta) \le \log \left(\frac{1+\delta}{\delta}\right).
$$
Using again the fact that $\log(1+\delta)>\delta/2$ holds for $\delta$ small enough, and similarly $\log((1+\delta)/\delta)\le-2\log\delta$, we have
$$
c_2\le \frac{-4\log \delta}{\delta}.
$$
By the same interval counting argument,
\begin{equation}\label{eJuly}
\sum_{j\in {\cal C}} \xi_2(j)\le \frac{-4 (m+r)\log \delta}{\delta}.
\end{equation}
Similarly, since $j-t\le\xi_2(j)$, equation \eqref{eJuly}
shows that the $C$ positive integers $j-t$ have total sum
at most $-4(m+r)\log\delta/\delta$.
Counting their possible partial sums, we find that the
number of possible lists of indices $t$ is at most
\begin{equation}\label{new2}
\binom{\lfloor-4(m+r)\log\delta/\delta\rfloor}{C}
\le
\left(\frac{-4e(m+r)\log\delta}{\delta C}\right)^C.
\end{equation}
From inequalities \eqref{new1} and \eqref{new2}, we conclude that the number of ways of choosing integers added in steps belonging to ${\cal C}$ is no larger than
$$
\left(\frac{-8e^2\log 2(m+r)^2\log \delta}{(\delta C)^2}\right)^{C},
$$
thus completing the proof of Lemma \ref{ub8}.
\end{proof}

Finally, we now focus on the number of ways of choosing the integers added in steps belonging to ${\cal D}$. We have the following.
\begin{lemma}\label{ub9}
The number of ways of choosing the integers added in steps belonging to ${\cal D}$ is smaller than or equal to
$$
\left(\frac{e(m+r+1)r}{(1-\log_2 \gamma)D}\right)^D.
$$
\end{lemma}

\begin{proof}
For $j\in {\cal D}$, write $a_j=a_t+a_s$ with $j-1\ge t\ge s$.  Then $a_t\ge a_j/2$, otherwise $a_t+a_s\le 2a_t<a_j$. Suppose that $a_{t+1}=2a_t$, in which case $t<j-1$, otherwise $j$ would be a doubling step, implying that $a_t= a_{t+1}/2<a_j/2$, a contradiction.   The pair $(s,t+1)$ must therefore belong to the product set
$$
\left({\cal A}\cup {\cal B}\cup {\cal C}\cup {\cal D}\cup \{0\}\right) \times \left({\cal B}\cup {\cal C}\cup {\cal D}\right).
$$
In other words, in an additive step, the largest added integer may not precede a doubling step.
The cardinality of the set $\left({\cal A}\cup {\cal B}\cup {\cal C}\cup {\cal D}\cup \{0\}\right) \times \left({\cal B}\cup {\cal C}\cup {\cal D}\right)$ is equal to
$$
(m+r+1)(B+C+D),
$$
which we can bound above using Lemma \ref{ub2} by
$$
\frac{(m+r+1)r}{1-\log_2 \gamma}.
$$
The number of ways of choosing $D$ pairs of indices $(s,t+1)$ among the set \\ \hbox{$\left({\cal A}\cup {\cal B}\cup {\cal C}\cup {\cal D}\cup\{0\}\right) \times \left({\cal B}\cup {\cal C}\cup {\cal D}\right)$} is thus at most
$$
{\lfloor (m+r+1)r/(1-\log_2 \gamma)\rfloor\choose D}.
$$
We conclude that the number of ways of choosing integers added in steps belonging to the set ${\cal D}$ is smaller than or equal to
$$
 {\lfloor (m+r+1)r/(1-\log_2 \gamma )\rfloor \choose D}\le \left(\frac{e(m+r+1)r}{(1-\log_2 \gamma)D}\right)^D,
$$
thus completing the proof of Lemma \ref{ub9}.
\end{proof}

We are now ready to complete the proof of the upper bound stated in Theorem \ref{t1}.
Gathering our results, namely Lemmas \ref{ub1}, \ref{ub4}, \ref{ub5}, \ref{ub8} and \ref{ub9}, we obtain that
\begin{eqnarray}\label{ll}
G(m,r)&\le& (m+r)^3 4^{r/(1-\log_2 \gamma)}e^{2r/(e(1-\log_2 \gamma))}\left(\frac{e(m+r)}{C+D}\right)^{C+D}\nonumber\\
& &\times \left(\frac{-8e^2\log 2(m+r)^2\log \delta}{(\delta C)^2}\right)^{C}\left(\frac{e(m+r+1)r}{(1-\log_2\gamma)D}\right)^D,
\end{eqnarray}
where the right-hand side is to be understood as the maximum possible value for the valid choices of $C$ and $D$. Assuming $\log m=o(r)$ and $r<m$, we deduce from (\ref{ll}) that there exists a positive computable constant $K$ such that for $m$  sufficiently large,
\begin{equation}\label{filon4}
G(m,r)\le K^r\left(\frac{-m^3\log \delta}{(\delta C)^2(C+D)}\right)^{C}\left(\frac{m^2r}{D(C+D)}\right)^D.
\end{equation}

To conclude, we need a technical lemma that will provide constraints on $C$ and $D$.
\begin{lemma}\label{ub10}
The cardinalities $C$ and $D$ must satisfy the inequality
$$
D\le \frac{r-C(1-\log_2\gamma)}{1-\delta/\log 2}.
$$
\end{lemma}
\begin{proof}
This is Lemma 4.2 in \cite{DKN}.
\end{proof}
Using Lemma \ref{ub10} and inequality \eqref{filon4}, it follows that
\begin{equation}\label{barb}
G(m,r)\le K^r\left(\frac{-m^3\log \delta}{\delta^2 C^3}\right)^{C} \left(\frac{ m^2r}{(r-C(1-\log_2\gamma))^2}\right)^{\frac{r-C(1-\log_2\gamma)}{1-\delta/\log 2}}.
\end{equation}
We have
\begin{eqnarray*}
& &\left(\frac{ m^2r}{(r-C(1-\log_2\gamma))^2}\right)^{\frac{r-C(1-\log_2\gamma)}{1-\delta/\log 2}}\\
&=&\left(\frac{m^2}{r}\right)^{\frac{r-C(1-\log_2\gamma)}{1-\delta/\log 2}}\left(\frac{r^2}{(r-C(1-\log_2\gamma))^2}\right)^{\frac{r-C(1-\log_2\gamma)}{1-\delta/\log 2}}\\
&\le &\left(\frac{m^2}{r}\right)^{\frac{r-C(1-\log_2\gamma)}{1-\delta/\log 2}} e^{2r}.
\end{eqnarray*}
Using this inequality in \eqref{barb}, we obtain
\begin{eqnarray}
G(m,r) &\le& (e^2K)^r\left(\frac{-m^3\log \delta}{\delta^2 C^3}\right)^{C}\left(\frac{m^2}{r}\right)^{\frac{r-C(1-\log_2\gamma)}{1-\delta/\log 2}} \nonumber\\
&=& (e^2K)^r \left(\frac{m^2}{r}\right)^{\frac{r}{1-\delta/\log 2}} \nonumber \\
& & \qquad \times \left(\frac{-m^3\log \delta}{\delta^2 C^3}\right)^{C}\left(\frac{m^2}{r}\right)^{\frac{-C(1-\log_2\gamma)}{1-\delta/\log 2}}.\label{novel1}
\end{eqnarray}
In order to find an upper bound for this last expression, we relax the condition that $C$ has to be an integer and look for the maximum value of the quantity
\begin{equation}\label{filon}
C\log\left(\frac{-m^3\log \delta}{\delta^2 C^3}\right)-\frac{C(1-\log_2\gamma)}{1-\delta/\log 2}\log(m^2/r)
\end{equation}
considered as a function of $C$, where $C$ is treated as a real number. Taking the derivative of \eqref{filon} and equating it to zero yields the equation
\begin{equation}\label{filon2}
\log\left(\frac{-m^3\log \delta}{\delta^2 C^3}\right)-\frac{(1-\log_2\gamma)}{1-\delta/\log 2}\log(m^2/r)-3=0.
\end{equation}
Solving equation (\ref{filon2}) for $C$, we get that
$$
\frac{-m^3\log \delta}{\delta^2 C^3} =e^3 \left(\frac{m^2}{r}\right)^{(1-\log_2\gamma)/(1-\delta/\log 2)},$$
implying that
$$C^3=\left(\frac{-e^{-3}\log \delta}{\delta^2}\right)\frac{m^{3-(2-2\log_2\gamma)/(1-\delta/\log 2)}}{r^{-(1-\log_2\gamma)/(1-\delta/\log 2)}}.
$$
From this and using the assumption $r<m$, we deduce
$$
C^3<\left(\frac{-e^{-3}\log \delta}{\delta^2}\right)m^{3-(1-\log_2\gamma)/(1-\delta/\log 2)}.
$$
It follows that the value of $C$ maximizing (\ref{filon}) satisfies $C<m^k$ for any $k>1-(1-\log_2 \gamma)/3\approx 0.8981$.

Using this in inequality (\ref{novel1}), it follows that
$$
G(m,r) \le  (e^2K)^r\left(\frac{-(\log \delta) m^{3}}{\delta^2}\right)^{m^k}\left(\frac{m^2}{r}\right)^{r(1+2\delta/\log 2)}.
$$
Choosing $\delta=1/\log m$ and observing that we therefore have $$ \frac{-\log \delta}{\delta^2} = (\log \log m)(\log m)^2 \le (\log m)^3,$$ this yields
\begin{eqnarray*}
G(m,r)&\le& (e^2K)^r\left((\log m)^3 m^{3}\right)^{m^k}\left(\frac{m^2}{r}\right)^r\left(\frac{m^2}{r}\right)^{2r/(\log m \log 2)}\\
&\le& (e^2Ke^{4/\log 2})^r\left( (\log m)^3 m^{3}\right)^{m^k}\left(\frac{m^2}{r}\right)^r.
\end{eqnarray*}
Provided that $r>m^{\eta}$ where $\eta$ is any constant strictly larger than $k$, this implies that there exists a computable constant $K_4$ such that for $m$ and $r$ sufficiently large,
$$
G(m,r)\le K_4^r\left(\frac{m^2}{r}\right)^{r},
$$
thereby  completing the proof of the upper bound in Theorem \ref{t1}.

\section{Proving the upper bound in Theorem \ref{t2}} 

Thus far, we have provided an upper bound for the number of addition chains. However, we did not take into account the possibility that several addition chains can lead to the same integer.  Doing so in this section will help us obtain a sharper upper bound for $F(m,r)$.
The proof of the upper bound in Theorem \ref{t2} will use some of the ideas developed in the proof of the upper bound in Theorem \ref{t1}. We begin by introducing two notions already used in \cite{DKN}.

\begin{definition}\label{d1}  Let $1=a_0<\cdots<a_{m+r}=n$ be an addition chain. Suppose that there exists another addition chain $1=b_0<\cdots<b_{m+r}=n$  of the same length and leading to the same integer $n$ and such that $b_j\le a_j$ for all $j$. Suppose also that there exists a $j$ for which $b_j<a_j$. Then we consider that the chain $a_0,\ldots,a_{m+r}$ is {\it invalid}. Otherwise, we say that the chain is {\it valid}.
\end{definition}

For example, the chain $1,2,4,5$ is invalid because of the existence of the chain $1,2,3,5$.
Intuitively, the idea behind Definition \ref{d1} is to enable us to select a canonical addition chain of minimal length among all the  addition chains of minimal length leading to $n$.

\begin{definition} Let $1=a_0<\cdots <a_{m+r}=n$ be an addition chain.  We define an {\it addition block} (of consecutive additive steps) as a maximal sequence of consecutive steps in ${\cal B}\cup {\cal C}\cup {\cal D}$. We also consider that $(0)$ corresponding to $a_0=1$ is an addition block.
\end{definition}

For example, if we consider the chain
$$
a_0=1,\ a_1=2,\ a_2=4,\ a_3=5,\ a_4=7, a_5=14, \ a_6=15,
$$
we have the three addition blocks $(0)$, $(3,4)$ and $(6)$.  We will say that the second block is of length 2 while the third block is of length 1.

Specifying the number of blocks and their lengths (that is, the block structure of an addition chain) provides constraints on the number of ways of choosing the sets ${\cal A}$, ${\cal B}$, ${\cal C}$ and ${\cal D}$ as is stated in the following lemma.

\begin{lemma}\label{t21}
Assume that an addition chain has $k$ addition blocks (excluding the block $a_0=1$). Assume also that the cardinalities $A$, $B$, $C$ and $D$ are fixed. Then,  the number of ways of choosing the sets ${\cal A}$, ${\cal B}$, ${\cal C}$ and ${\cal D}$ is smaller than or equal to
$$
{B+C+D-1\choose k-1}{A+1\choose k}3^{B+C+D}.
$$
\end{lemma}

\begin{proof}
Let $L_1,\ldots, L_k$ be the lengths of the different addition blocks.  Then by definition we have
$$
L_1+L_2+\cdots+L_k=B+C+D.
$$
Hence, provided  $k$, $B$, $C$ and $D$ are all fixed, the number of ways of choosing the lengths of the blocks is at most
$$
{B+C+D-1\choose k-1}.
$$
Moreover, the number of ways to choose the number of elements of ${\cal A}$ between consecutive addition blocks is equal to
$$
{A+1\choose k}.
$$
Note that since each addition block is separated by at least one element in ${\cal A}$, it follows that  $A\ge k-1$, thereby implying that $\displaystyle{{A+1\choose k}}$ is indeed well defined.

Finally, the number of ways of choosing how the elements of ${\cal B}\cup {\cal C}\cup{\cal D}$ are distributed  within the blocks is smaller than or equal to
$$
3^{B+C+D}.
$$
Gathering the above three bounds, the proof of Lemma \ref{t21} is complete.
\end{proof}

In the following, when an element $a_j$ of an addition chain can be expressed in several ways as $a_j=a_s+a_t$ with $s\le t$ and $a_s,\, a_t$ previous members of the chain, we consider the choice for which $s$ is the smallest possible.

We consider the following definition relating addition blocks to each other.

\begin{definition} We say that an addition block is \emph{connected} to the block $a_0=1$ if $a_0=1$ is used in an addition step of this block.  We say that an addition block $B_1=(j,j+1,\ldots,j+t)$ is \emph{connected} to another addition block $B_2$ if one of the elements $a_{j-1},\ldots, a_{j+t}$ is used in an additive step of $B_2$. We extend this definition by transitivity and symmetry, meaning that if a block $B_1$ is connected to a block $B_2$, then $B_2$ is connected to $B_1$ and if $B_2$ is connected to a block $B_3$, then $B_1$ is connected to $B_3$. We can think of this notion as providing a graph structure on addition blocks.
\end{definition}

\begin{definition}
Consider an addition block $(j,\, j+1,\ldots, j+t)$ such that each of the integers $a_j, a_{j+1},\ldots, a_{j+t}$ is even.  We say that we perform a {\it left shift operation} on this block if in the addition chain we replace the sequence $a_{j-1}=2a_{j-2}, a_j, \ldots, a_{j+t}$ by the sequence $a_{j}/2,a_{j+1}/2,\ldots, a_{j+t}/2, a_{j+t}$.  Remark that after performing this operation, the resulting sequence of integers may no longer be an addition chain. 
\end{definition}

For example, consider an addition chain beginning with $1,2,4,8,12,14,18,22$. This chain has an addition block with corresponding integers $12,14,18,22$. Performing a left shift operation on this block transforms the addition chain to  $1, 2, 4,\mathbf{6, 7, 9, 11},  22$.

With this definition we can show the following lemma. 
 
\begin{lemma}\label{t23} Suppose that an addition chain contains an addition block that is not connected to the block corresponding to $a_0=1$. Then the addition chain is not valid.
\end{lemma}

\begin{proof} Assume that there are $s$ addition blocks, say $B_1,B_2,\ldots, B_s$, that are not connected to $a_0=1$, and that these blocks are ordered increasingly. One can first observe that all the elements of all the blocks are even. For $B_1$, this follows from the fact that each element in the block is obtained from adding elements in ${\cal A}$ or elements in the block $B_1$. The fact that the elements of other blocks are even can be shown recursively. If we perform a left shift operation on each of these blocks,
the resulting sequence is still an addition chain. Indeed, consider a chosen addition $a_i=a_s+a_t$ in a selected block. A parent belonging to a selected block has its half already inserted. A parent with a doubling index $h$ has half $a_{h-1}$,
which is unchanged, since every modified index immediately precedes a nondoubling step. Thus
$$
a_i/2=a_s/2+a_t/2
$$
is a sum of two earlier terms of the new sequence.

No chosen addition in an unselected block uses a changed term, by the definition of connections. Each retained block endpoint is obtained by doubling its new predecessor, and unchanged
doubling steps retain their predecessors. The new sequence increases strictly. Every changed coordinate decreases strictly, since $a_{i+1}/2<a_i$ whenever $i+1$ is an additive step. We have therefore obtained a smaller chain of the same length and endpoint, contradicting validity.
\end{proof}





We divide the set ${\cal D}$ (or ${\cal C}$) into two disjoint subsets ${\cal D}={\cal D}_1\cup {\cal D}_2$  (or ${\cal C}={\cal C}_1\cup {\cal C}_2$). If $j\in {\cal D}$ (or $j\in {\cal C}$), we write $a_j=a_s+a_t$ with $t\le s$.  We say that $j\in {\cal D}_1$ (or $j\in {\cal C}_1$) if either $t$ or $t+1$ can belong to the set ${\cal B}\cup {\cal C}\cup {\cal D}\cup \{0\}$. Otherwise, we say that $j\in {\cal D}_2$ (or ${\cal C}_2$, respectively). We denote by $D_1$ and $D_2$ (or $C_1$ and $C_2$) the respective cardinalities of these sets.  In other words, an additive step belongs to ${\cal C}_1$ or ${\cal D}_1$ if an element belonging to ${\cal B}\cup {\cal C}\cup {\cal D}\cup\{0\}$ can be used as the smaller integer in the corresponding addition.

Because of Lemma \ref{t23}, all the blocks are connected to $a_0=1$, hence there must be at least $k$ additive steps making connections between blocks. It follows that $C_1+D_1\ge k$ from which we get the following lemma. 

\begin{lemma}\label{ub12}
Let $1=a_0<\cdots<a_{\ell}=n$ be a valid addition chain with $k$ addition blocks (excluding the block $a_0=1$).  Then
$$
D_1\ge \max(k-C,0).
$$
\end{lemma}
Intuitively, the set ${\cal D}_1$ is more restrictive than the set ${\cal D}_2$ because the number of ways of choosing the integers added in a step belonging to ${\cal D}_1$ is smaller.  Lemma \ref{ub12} will thus enable us to improve the upper bound.


Following the steps of the proof of the upper bound of  Theorem \ref{t1}, we will show that the case $k\leq C$ can be neglected since it leads to only a small number of addition chains.

\begin{lemma}\label{ub13}
Let $k$ be the number of addition blocks in a valid addition chain (excluding the block $a_0=1$).  The number of ways of choosing integers added in steps belonging to ${\cal D}$ is smaller than or equal to
$$
{2\lfloor r/(1-\log_2 \gamma)\rfloor^2 \choose \max(0, k-C)}{\lfloor r/(1-\log_2 \gamma)\rfloor (m+r)\choose D-\max(0,k-C)}.
$$
\end{lemma}
 \begin{proof} 
For this count, select the first $\max(0,k-C)$ steps belonging to ${\cal D}_1$, whose existence follows from Lemma \ref{ub12}, and for each select a representation satisfying the defining condition of ${\cal D}_1$.
The number of ways of choosing the pairs of indices for these steps is smaller than or equal to
$$
{2(B+C+D)^2\choose \max(0,k-C)}.
$$
Meanwhile, the remaining $D-\max(0,k-C)$ pairs can be chosen from the full set used in Lemma \ref{ub9}, giving at most
$$
{(B+C+D)(m+r)\choose D-\max(0,k-C)}
$$
possibilities.
The number of ways of choosing integers added in steps belonging to ${\cal D}$ is thus smaller than or equal to
\begin{eqnarray*}
&& {2(B+C+D)^2\choose \max(0,k-C)}
{(B+C+D)(m+r)\choose D-\max(0,k-C)}\\
&& \quad \le
{2\lfloor r/(1-\log_2\gamma)\rfloor^2\choose \max(0,k-C)}
{\lfloor r/(1-\log_2\gamma)\rfloor(m+r)
 \choose D-\max(0,k-C)},
\end{eqnarray*}
where again we made use of Lemma \ref{ub2}. Once the choices at the other steps have been fixed, the union of the two sets of recorded pairs determines their order of use: at each step belonging to ${\cal D}$, use the unused pair whose indices are already available and whose corresponding terms
have the smallest sum.
  \end{proof}

We are now ready to complete the proof of the upper bound in Theorem \ref{t2}. If $r\ge m\log 2/\log m$, the desired bound follows from
$
F(m,r)\le 2^m\le m^r.
$
We may therefore assume that $r<m\log 2/\log m$, and hence
$r=o(m)$. Using Lemmas \ref{ub1}, \ref{ub5}, \ref{ub8}, \ref{t21} and \ref{ub13}, we have that
\begin{eqnarray}
F(m,r)&\le& K^r{m+r+1\choose k}\left(\frac{-8e^2\log 2(m+r)^2\log\delta}{(\delta C)^2}\right)^C\nonumber\\
& &\times {2\lfloor r/(1-\log_2 \gamma)\rfloor^2 \choose \max(0, k-C)}{\lfloor r/(1-\log_2 \gamma)\rfloor(m+r) \choose D-\max(0,k-C)} \nonumber\\
&\le& (K')^r \left(\frac{m}{k}\right)^k\left(\frac{m^2(\log m)^3}{C^2}\right)^C{r^2\choose \max(0,k-C)}\nonumber \\
& &\times{rm\choose D-\max(0,k-C)}\label{2026}
\end{eqnarray}
for some positive computable constants $K$ and $K'$, and where the right-hand side is to be understood as the largest possible value over suitable choices of $k$, $C$ and $D$. If $k\leq C$ and under the assumption $r=o(m)$, inequality \eqref{2026} implies that for some constant $K''>0$ and for $m$ sufficiently large,
$$
F(m,r)\le (K'')^r\left(\frac{m^3(\log m)^3}{C^3}\right)^C\left(\frac{rm}{D}\right)^{D}.
$$
Using the bound
$$
D\le \frac{r-C(1-\log_2\gamma)}{1-\delta/\log 2}
$$
given in Lemma \ref{ub10}, we obtain from the above inequality that
\begin{eqnarray*}
F(m,r)&\le& (K'')^r\left(\frac{m^3(\log m)^3}{C^3}\right)^Cm^D\left(\frac{r}{D}\right)^{D}\\
&\le& (K''')^r\left(\frac{m^{2+\log_2\gamma}(\log m)^3}{C^3}\right)^Cm^r\\
&\le& (2K''')^r m^r
\end{eqnarray*}
for some positive constant $K'''$. Here we used that $(2+\log_2\gamma)/3<0.9$ and $r>m^{0.9}$. We can thus assume that $k>C$ in inequality \eqref{2026}. Under this hypothesis, we have
\begin{eqnarray*}
F(m,r)&\le& (K')^r \left(\frac{m}{k}\right)^k\left(\frac{m^2(\log m)^3}{C^2}\right)^C{r^2\choose k-C}{rm \choose D+C-k}\\
&\le& (K'')^r\left(\frac{m}{k}\right)^k\left(\frac{ m^2(\log m)^3}{C^2 }\right)^C \left(\frac{r^2}{k-C}\right)^{k-C}\left(\frac{rm}{D+C-k}\right)^{D+C-k}.
\end{eqnarray*}
From this, we obtain
\begin{eqnarray*}
F(m,r)&\le& (K'')^r \left(\frac{m}{k}\right)^k\left(\frac{ m^2(\log m)^3}{C^2 }\right)^C\left(\frac{r^2}{k}\right)^{k-C} m^{D+C-k}\\
& &\times \left(\frac{k}{k-C}\right)^{k-C}\left(\frac{r}{D+C-k}\right)^{D+C-k}\\
&\le&(K''')^r \left(\frac{m}{k}\right)^k\left(\frac{ m^2(\log m)^3}{C^2 }\right)^C\left(\frac{r^2}{k}\right)^{k-C} m^{D+C-k}\\
&=&(K''')^r \left(\frac{r^2}{k^2}\right)^k\left(\frac{ m^3(\log m)^3k}{r^2C^2 }\right)^Cm^D.
\end{eqnarray*}
Using Lemma \ref{ub10}, we have
$$
D\le
\frac{r-C(1-\log_2\gamma)}{1-\delta/\log 2}
\le
r-C(1-\log_2\gamma)+\frac{2r\delta}{\log 2}.
$$
Since $\delta=1/\log m$, it follows that
$$
m^D\le
\exp\Big(\frac{2r}{\log 2}\Big)
m^r m^{-C(1-\log_2\gamma)}.
$$
Using this in the previous inequality, we get
\begin{eqnarray*}
F(m,r)&\le&(K'''')^r \left(\frac{m^{2+\log_2\gamma}(\log m)^{3}}{rC^2}\right)^Cm^r\\
&\le& (K'''')^r(m^3)^{
m^{1+(\log_2\gamma)/2}(\log m)^{3/2}/\sqrt r
}
m^r.
\end{eqnarray*}
The last inequality follows by considering separately whether
$$
C\ge
\frac{m^{1+(\log_2\gamma)/2}(\log m)^{3/2}}{\sqrt r}
$$
or not, as in the first case.

Finally,
$$
\frac{
m^{1+(\log_2\gamma)/2}(\log m)^{5/2}
}{r^{3/2}}
=o(1),
$$
because $r>m^{0.9}$ and $(2+\log_2\gamma)/3<0.9$.
Thus the additional power of $m^3$ is $\exp(o(r))$.
After enlarging the positive computable constant $K_2$,
both cases give
$$
F(m,r)\le K_2^r m^r,
$$
which completes the proof of the upper bound in
Theorem \ref{t2}.

\section{Proving the lower bound in Theorem \ref{t2}}

The proofs of the lower bounds for $F(m,r)$ and $G(m,r)$ will be constructive ones.   For the sake of  exposition, we begin by recalling the method used by Brauer to prove his upper bound for $\ell(n)$.

\subsection{The Brauer approach}

In \cite{brauer}, for a fixed large integer $m$, Brauer provided a constructive way to build addition chains leading to any integer $n$  between $2^m$ and $2^{m+1}$. From this, an upper bound for $\ell(n)$ can be deduced.

The constructive process is described as follows.

\begin{enumerate}
\item Choose a positive integer $k$ such that $k<m$,
\item Start the chain by additive steps to obtain $1,2,3,\ldots, 2^k-1$,
\item Alternate $m/k$ sequences of $k$ doubling steps followed by the addition of an integer obtained in Part 2.
\end{enumerate}

In Part 3 of this process, we implicitly assumed that $k$ divides $m$, but the more general case is not difficult to deduce.  It is relatively easy to see that any integer $n\in (2^m,2^{m+1})$ can be generated through this process. The number of steps in Parts 2 and 3 of the process is at most
$$
2^{k}+m+\frac{m}{k},
$$
which shows that for $n<2^{m+1}$,
\begin{equation}\label{brauer}
\ell(n)\le \min_{1\le k \le m} \left\lfloor2^{k}+m+\frac{m}{k}\right\rfloor.
\end{equation}
To obtain an upper bound, it suffices to optimize the choice of $k$. Observe that equation (\ref{brauer}) is morally the result stated by Brauer \cite{brauer} and recalled in the Introduction.

\subsection{A constructive lower bound for \texorpdfstring{$F(m,r)$}{F(m,r)}}

We now provide a constructive proof of the lower bound in
Theorem \ref{t2}. Throughout this subsection, fix $0<c<\log 2$
and assume that
$$
m^{0.9}<r\le\frac{cm}{\log m}.
$$
All asymptotic estimates below hold uniformly over the integers
$r$ in this range as $m$ tends to infinity, with $c$ fixed. The threshold for $m$ may depend on $c$.

This proof will be divided into two distinct stages. In the first stage (parts 1--4), we describe an algorithm to construct addition chains, while in the second stage (parts 5--8) we prove a lower bound for the number of distinct integers that can be obtained by our construction. More precisely:

\begin{enumerate}
\item We give a sequence of choices that produces many addition chains of
length at most $m+r$.  
Our construction divides an addition chain into {\it blocks}. Each block will typically contain doubling steps and two additive steps.
The total number of ways of constructing all the blocks will be at least
$
 m^re^{-O_c(r)}.
$

\item A difficulty will arise from the fact that different constructions (different sequences of choices) may lead to the same final integer.
However, using graphs related to the addition chains, we will restrict ourselves to a subset of {\it good} constructions.  We will then show that most sequences of choices lead to good constructions and that no final integer is produced by more than $e^{o(r)}$ good constructions. Together with our bound on the number of constructions, this will prove the result.

\end{enumerate}

\vskip 10pt

\noindent
{\bf Part 1. The setup}

Put
$$
\rho_0=\frac{\log 2}{c}>1,
\qquad
\rho=\frac{m}{r\log_2r}.
$$
Since $r\le cm/\log m$ and $r<m$ for sufficiently large $m$,
we have $\rho\ge\rho_0$. Moreover,
\begin{equation}
\frac mr=\rho\log_2r.
\label{eq:F-new-m-over-r}
\end{equation}
We use the two parameters
\begin{equation}
L=\left\lfloor\frac{\rho-1}{16}\log_2r\right\rfloor,
\qquad
K=\left\lceil\frac{\rho+3}{2}\log_2r\right\rceil.
\label{eq:F-new-parameters}
\end{equation}
Thus $L,K\asymp_c m/r$, and $K>4L$ for all sufficiently
large $m$. We will also use $\log r\asymp\log m$.

At any point, the chain constructed so far has the form
$$
 1=a_0<a_1<\cdots<a_s=X,
$$
where $X$ is its current last term.  We keep a subset $\mathcal H$ of
earlier odd terms among $a_0,\ldots,a_s$, which we call a 
\emph{reservoir}. A \emph{phase} is a consecutive collection of blocks that all use the same fixed
reservoir, called the \emph{input reservoir} of the phase. Outputs created
during a phase are not used as sources until the phase has ended. In each block, we use two reservoir elements:
\begin{itemize}
\item $H_p$, called the \emph{principal source};
\item $H_q$, called the \emph{secondary source}.
\end{itemize}
These names refer to the role of the reservoir elements in the current block. Thus the same reservoir element may be principal in
one block and secondary in another.  In block $i$, the two sources are
denoted by $H_{p_i}$ and $H_{q_i}$.

Suppose a phase has been assigned a nonnegative integer $R$, called its
\emph{separator}, and its input reservoir is
$$
 \mathcal H=\{H_1,\ldots,H_N\}.
$$
The separator is simply a long run of doublings at the first block. Later it
will let us recover the value at the beginning of the phase from the value at
its end.
In block $i$, choose a principal source $H_{p_i}$, a secondary source
$H_{q_i}$, and integers
$$
 e_i\in\{L,L+1,\ldots,2L-1\},
 \qquad
 u_i\in\{0,1,\ldots,L-1\}.
$$
Starting from the current last term $X_{i-1}$, we will perform four
operations to get the next term
\begin{equation}
 X_i=
 2^{K+u_i+R\mathbf 1_{\{i=1\}}+1}X_{i-1}
 +2^{e_i}H_{q_i}+H_{p_i},
 \label{eq:F-new-block}
\end{equation}
which we call the
\emph{output} of that block.
Every block therefore has exactly two nondoubling steps, and in particular
\begin{equation}
 X_i\equiv H_{p_i}+2^{e_i}H_{q_i}\pmod {2^K}.
 \label{eq:F-new-low-bits}
\end{equation}
Here and below, the last $K$ binary digits of an integer mean its remainder
modulo $2^K$.
Thus, the last $K$ binary digits of $X_i$ do not depend on $X_{i-1}$ or on
$u_i$.

If no restrictions were imposed, a block based on a reservoir of size $N$ would have $N^2L^2$ choices for $(p,q,e,u)$ from which we could deduce the theorem. The construction below imposes some mild restrictions on these choices, but only loses a constant or subexponential factor.

\vskip 10pt

\noindent
{\bf Part 2. The initial reservoir}

We next construct the initial reservoir. Let
$$
N_0=\left\lfloor\frac{r}{K(\log_2r)^2}\right\rfloor,
$$
and choose $N_0$ arbitrary odd integers
$$
 1\leq U_s<2^K,
 \qquad 1\leq s\leq N_0.
$$
Starting with $Z_0=1$, append terms so that
\begin{equation}
 Z_s=2^{K+2L}Z_{s-1}+U_s,
 \qquad 1\leq s\leq N_0.
 \label{eq:F-new-seeds}
\end{equation}
To do this, first double $Z_{s-1}$ exactly $2L$ times.  
Then read the $K$ binary
digits of $U_s$: at each digit, double the current term and, when the digit is
$1$, add the earlier term $1$.  This uses $K+2L$ doublings and at most $K$
nondoubling steps.

For example, take $L=1$, $K=6$, and $U_1=3$.  Since the six-bit expansion of
$3$ is $000011$, the construction begins
$$
 1,2,4,8,16,32,64,128,\mathbf{129},258,\mathbf{259}=Z_1.
$$
The bold terms are the two nondoubling steps, and indeed
$Z_1=2^8+3=259\equiv3\pmod {64}$.

The \emph{initial reservoir} is
$$
 \mathcal H_0=\{Z_1,\ldots,Z_{N_0}\}.
$$
Its elements, called the \emph{seeds}, are odd and satisfy
\begin{equation}
 Z_s\equiv U_s\pmod {2^K}.
 \label{eq:F-new-seed-residues}
\end{equation}
The number of nondoubling steps used in constructing the seeds is at most
\begin{equation}
s_0:=N_0K\le\frac{r}{(\log_2r)^2}
=O\left(\frac{r}{(\log r)^2}\right).
\label{eq:F-new-seed-cost}
\end{equation}

All the shifted (multiplied by powers of two) secondary sources required later are genuine earlier chain
terms.  For $Z_s$ with $s<N_0$, the necessary shifts are created at the start
of the construction of $Z_{s+1}$.  The binary shifts of $Z_{N_0}$ are created by the
initial doublings of the first block.  Likewise, every output of a growth
phase except its last is doubled at least $K$ times at the start of the next
block, while the last output is doubled at the start of the next phase.  Since
$e_i-1<2L<K$, this supplies every term $2^{e_i-1}H_{q_i}$ used in a block.

\vskip 10pt

\noindent
{\bf Part 3. Growing the reservoir and choosing the records}

Suppose $\mathcal H_j$ has size $N_j$.  A \emph{growth phase} based on
$\mathcal H_j$ consists of $4N_j$ blocks.  Every element of $\mathcal H_j$ is
used exactly four times as principal source, in an arbitrary order, and the
odd outputs form
$$
 \mathcal H_{j+1}=\{X_1,\ldots,X_{4N_j}\},
 \qquad
 N_{j+1}=4N_j.
$$

Let $\alpha=\alpha(c)>0$ be a sufficiently small constant, to be chosen in
the length calculation, and assume $\alpha<1/8$.  Let $J$ be the largest integer such
that
$
 N_J=4^JN_0\leq\alpha r.
$
Then
\begin{equation}
 \frac{\alpha r}{4}<N_J\leq\alpha r,
 \qquad
 \sum_{j<J}4N_j\leq\frac43N_J.
 \label{eq:F-new-reservoir-sums}
\end{equation}
We use
\begin{equation}
 M_{\rm all}=\left\lfloor\frac{r-s_0}{2}\right\rfloor
 \label{eq:F-new-total-blocks}
\end{equation}
blocks in total.  After the growth phases, the remaining
$$
 M_*=M_{\rm all}-\sum_{j<J}4N_j
$$
blocks form a terminal phase based on $\mathcal H_J$.  For all sufficiently
large $r$, we have $M_*>r/4$.  In the terminal phase, every element of
$\mathcal H_J$ is used as principal source either
$\lfloor M_*/N_J\rfloor$ or $\lceil M_*/N_J\rceil$ times.  These
multiplicities are bounded in terms of $c$.

We now impose one simple rule on the secondary sources.  Give seed $Z_s$ the
root $s$, and let every output inherit the root of its principal source.  When
a block has principal root $a$ and secondary root $b$, draw a directed arrow
$$
 a\longrightarrow b
$$
labelled by the exponent $e$ used in that block.  Arrows persist throughout
the construction.  We require $a\neq b$ and forbid reuse of the same ordered
pair $a\to b$.  The reverse arrow $b\to a$ is allowed.

For example, suppose that the initial seeds have roots $1,2,3,4$.  If a
block uses principal root $1$ and secondary root $2$, it creates the arrow
$$
 1\xrightarrow{e}2,
$$
where $e$ is the exponent chosen in that block.  The output of the block
inherits root $1$.  Later, when the principal root is again $1$, root $2$
cannot be used as the secondary root, because that would reuse the arrow
$1\to2$.  Root $1$ is also forbidden, because loops are not allowed, but
roots $3$ and $4$ remain available.
The direction of the arrow matters:  a later block with principal root $2$
may still use secondary root $1$, thereby creating the different arrow
$$
 2\xrightarrow{e'}1.
$$
Thus the graph records ordered principal--secondary interactions: an arrow
leaving a root prevents only that same ordered interaction from being used
again.

For now, this rule guarantees many choices for the secondary source.  Later,
the arrows will serve as markers that distinguish different source histories.

Every reservoir $\mathcal H_j$ contains exactly $N_j/N_0$ elements of each
root.  Indeed, this is true of $\mathcal H_0$, and every old element is used
four times as principal source while each output inherits its principal
root.  Moreover, any fixed root is used as a principal root at most
$$
 \sum_{j<J}\frac{4N_j}{N_0}
 +\left\lceil\frac{M_*}{N_J}\right\rceil\frac{N_J}{N_0}
 =O_c(r/N_0)
$$
times in the whole construction.  Since $K=O_c(m/r)$ and $r>m^{0.9}$, the quantity
$r/(K(\log_2r)^2)$ tends to infinity uniformly. Consequently,
$$
\frac{r}{N_0^2}
=O_c\left(\frac{K^2(\log r)^4}{r}\right)
=O_c\left(\frac{m^2(\log m)^4}{r^3}\right)
=o(1).
$$
Thus $r/N_0=o(N_0)$. At a block whose principal root is $a$,
only $a$ and the $o(N_0)$ heads of arrows already leaving $a$
are forbidden as secondary roots. A phase based on a reservoir of size $N$
therefore has at least
\begin{equation}
 \frac N2
 \label{eq:F-new-secondary-count}
\end{equation}
choices for the secondary source in every block, for all sufficiently large
$r$.

It remains to restrict the extra-doubling choices so that their sum in each
phase is known.  In a phase of $M$ blocks, the $L^M$ possible sequences
$(u_1,\ldots,u_M)$ have only $M(L-1)+1\leq ML$ possible sums.  By the
pigeonhole principle, some sum $t$ is attained by at least $L^M/(ML)$
sequences.  Complementing every coordinate by $u_i\mapsto L-1-u_i$ shows
that the complementary sum $M(L-1)-t$ is attained equally often.  We fix the
smaller of these two sums.  Doing this separately in every phase gives
\begin{equation}
 \sum_{i=1}^{M_{\rm all}}u_i
 \leq\frac{M_{\rm all}(L-1)}2.
 \label{eq:F-new-u-budget}
\end{equation}
The number of phases is now $O_c(\log m)$. Indeed,
$J+1=O_c(\log(K(\log_2r)^2))$ by the choice of $N_0$
and the geometric growth of the reservoirs.
Since $rL\le m/16$, the product of the losses $ML$ is at most
$$
(rL)^{O_c(\log m)}
=\exp\big(O_c((\log m)^2)\big)
=e^{o(r)}.
$$
Thus the sum in every phase is fixed before any record is chosen, while the
total number of allowed $u$-sequences remains
$L^{M_{\rm all}}e^{-o(r)}$.

Finally, assign a deterministic separator to every phase.  For a growth phase
based on $\mathcal H_j$, and for the terminal phase based on $\mathcal H_J$,
put
\begin{equation}
 R_j=2(K+L+1)N_j.
 \label{eq:F-new-separators}
\end{equation}
This is the value of $R$ used in the first block of that phase.

We call a \emph{choice record} the order of the principal sources and all
the choices of $q_i,e_i,u_i$ in every phase.  Note that all the restrictions defining a
choice record depend only on the roots and on the preceding formal choices,
not on the numerical values of the $U_s$. The set of choice records is therefore the same for every allowed choice of the $U_s$. Counting the possible number of choice records will be the next step in the proof of the theorem.

\vskip 10pt

\noindent
{\bf Part 4. Counting records and controlling the length}

We first count the choice records, without yet asking whether two of them have the same endpoint.  
Consider a growth phase based on a reservoir of size $N$, with $M=4N$
blocks.  The number of possible orders of its principal sources is
$$
 \frac{(4N)!}{(4!)^N}\geq(c_0N)^M
$$
for an absolute constant $c_0>0$.  At every block there are at least $N/2$
secondary choices by \eqref{eq:F-new-secondary-count} and $L$ choices for
$e_i$.  The same principal-order estimate, with a constant depending on $c$,
holds in the terminal phase because its principal multiplicities are bounded.

For $0\leq j<J$, put
$$
 M_j=4N_j,
$$
and put $M_J=M_*$ for the terminal phase.  Thus $M_j$ is the number of
blocks in the phase whose input reservoir is $\mathcal H_j$, and
$$
 \sum_{j=0}^J M_j=M_{\rm all}.
$$
Combining the principal-source orders, the choices of secondary sources and
exponents, and the allowed extra-doubling sequences, we obtain, for the
total number $T$ of choice records,
$$
 \log T
 \geq
 2\sum_{j=0}^J M_j\log(N_jL)-O_c(r).
$$
Replacing every $N_j$ by $N_J$ costs only $O(N_J)$ in the logarithm, because
$$
 \sum_{j<J}4N_j\log\frac{N_J}{N_j}=O(N_J)=O_c(r).
$$
It follows that
$$
 \log T\geq2M_{\rm all}\log(N_JL)-O_c(r).
$$
For all sufficiently large $m$, the parameter choices give
$$
L\ge\frac{(\rho-1)\log_2r}{32}
\ge\frac{1-\rho_0^{-1}}{32}\frac mr.
$$
Together with \eqref{eq:F-new-reservoir-sums}, this implies
$$
N_JL\ge\frac{\alpha(1-\rho_0^{-1})}{128}m.
$$ Moreover,
$$
 (r-2M_{\rm all})\log m
 \leq(s_0+2)\log m=o(r)
$$
by \eqref{eq:F-new-seed-cost} and \eqref{eq:F-new-total-blocks}.
Consequently,
\begin{equation}
 T\geq m^re^{-O_c(r)}.
 \label{eq:F-new-T}
\end{equation}

We next verify the deterministic size bounds used in defining the separators.
The seed recurrence \eqref{eq:F-new-seeds} gives
$$
 Z_{N_0}<2^{N_0(K+2L+1)}<2^{R_0}.
$$
Suppose a growth phase based on $\mathcal H_j$ begins with an integer smaller than $2^{R_j}$.
Every source is then also below $2^{R_j}$.  During the phase, the separator
contributes $R_j$ doublings.  Since the sources are already present when the
phase begins, $H_{p_i},H_{q_i}\leq X_{i-1}$.  Consequently,
$$
 2^{e_i}H_{q_i}+H_{p_i}
 \leq(2^{2L-1}+1)X_{i-1}
 <2^{K+u_i+R_j\mathbf 1_{\{i=1\}}+1}X_{i-1},
$$
where we used $K>2L+1$.  By \eqref{eq:F-new-block}, block $i$ therefore
increases the current endpoint by a factor smaller than $2^{K+L+1}$, apart
from the separator $R_j$ in its first block.  The phase therefore ends below
$$
 2^{2R_j+4N_j(K+L+1)}.
$$
By \eqref{eq:F-new-separators},
$$
 2R_j+4N_j(K+L+1)=4R_j=R_{j+1}.
$$
Induction therefore proves that every growth phase begins with a smaller integer than its prescribed
bound $2^{R_j}$, and the terminal phase begins below $2^{R_J}$.  In addition,
\begin{equation}
 \sum_{\text{phases}}R_j
 \leq\frac83(K+L+1)N_J
 \leq\frac83\alpha(K+L+1)r.
 \label{eq:F-new-separator-total}
\end{equation}

Let $D_0$ be the total number of doubling steps used up to
the final odd endpoint $X$. Since $2L\le K/4$, the
construction of the seeds uses
$$
N_0(K+2L)\le\frac{5r}{4(\log_2r)^2}=o(r)
$$
doubling steps. The blocks contribute $M_{\rm all}K$
compulsory doubling steps, one further doubling step per
block, the extra doubling steps $u_i$, and the separators.
Therefore, \eqref{eq:F-new-u-budget} and
\eqref{eq:F-new-separator-total} give
$$
\frac{D_0}{r}
\le\frac K2+\frac L4
 +\frac83\alpha(K+L+1)+O(1).
$$
From \eqref{eq:F-new-parameters},
$$
\frac K2+\frac L4
\le\left(1+\frac{17}{64}(\rho-1)\right)\log_2r+O(1).
$$
The difference between $\rho$ and the coefficient on the
right is $47(\rho-1)/64>0$. Choose $0<\alpha<1/8$,
depending only on $c$, so small that
$$
\frac83\alpha
\left(\frac{2}{\rho_0-1}+\frac9{16}\right)
<\frac{47}{128}.
$$

Let $A_1$ be the total number of nondoubling steps used to
obtain $X$. By \eqref{eq:F-new-total-blocks},
$$
A_1\le s_0+2M_{\rm all}\le r.
$$
Put
$$
\tau=\frac{47}{128}(\rho-1).
$$
Since $\rho\ge\rho_0$, our choice of $\alpha$ gives
$$
\frac83\alpha(K+L+1)\le\tau\log_2r+O(1).
$$
Combining these estimates, we obtain
$$
\frac{D_0}{r}\le(\rho-\tau)\log_2r+O(1).
$$
Using \eqref{eq:F-new-m-over-r}, $A_1/r\le1$, and
$$
\tau\ge\frac{47}{128}(\rho_0-1)>0,
$$
we conclude that
$$
D_0+A_1\le m
$$
for all sufficiently large $m$, uniformly in the stated range.

Put $S=\lfloor\log_2X\rfloor$.  Doubling steps alone give $D_0\leq S$, while
every nondoubling step increases the current last term by a factor smaller
than $2$.  Hence
$$
 D_0\leq S\leq D_0+A_{1}\leq m.
$$
Append $m-S$ doubling steps and define
\begin{equation}
 n=2^{m-S}X.
 \label{eq:F-new-normalization}
\end{equation}
Then $2^m\leq n<2^{m+1}$, and the completed chain has length at most
$$
 D_0+A_{1}+m-S\leq m+A_{1}\leq m+r.
$$
The map $X\mapsto n$ is injective because $X$ is the odd part of $n$.

We have now proved that, for every fixed choice of the seed residues, the same
set of $T\geq m^re^{-O_c(r)}$ choice records produces addition chains of length at most $m+r$. This concludes the first stage of the proof. It
remains to choose the seed residues so that not too many records have the same
odd endpoint.

\vskip 10pt

\noindent
{\bf Part 5. Recovering phases and individual blocks}

For a fixed seed-residue vector $\mathbf U$, by \emph{decoding} a final odd
endpoint we mean reconstructing all choice records that could have produced
that endpoint.  The \emph{decoder} is the backward reconstruction procedure
used for this purpose.  It need not produce a unique record: whenever several
source triples are compatible with the information already recovered, the
decoder retains all of them.

Consider a phase with $M$ blocks, separator $R$, and fixed sum
$$
 t=\sum_{i=1}^M u_i.
$$
Write
$$
 g_i=K+u_i+R\mathbf 1_{\{i=1\}}+1,
 \qquad
 Q_i=2^{e_i}H_{q_i}+H_{p_i}.
$$
If $X$ and $Y$ are the values at the beginning and end of the phase, then
iteration of \eqref{eq:F-new-block} gives
$$
 Y=2^GX+E,
 \qquad
 G=M(K+1)+R+t,
$$
where
$$
 E=\sum_{i=1}^M2^{g_{i+1}+\cdots+g_M}Q_i.
$$
The integer $G$ is known from the fixed list of phase sums.  Every source in
the phase is smaller than $2^R$, while $e_i<2L$, and hence
$$
 Q_i<2^{R+2L+1}.
$$
Since $g_1\geq K+R+1$ and $g_i\geq K+1$ for $i>1$, a geometric series
estimate gives
$$
 \frac E{2^G}
 <
 \frac{2^{R+2L+1}}{2^{K+R+1}}
 \frac1{1-2^{-(K+1)}}
 <2^{2L+1-K}<1.
$$
Therefore
\begin{equation}
 X=\left\lfloor\frac Y{2^G}\right\rfloor.
 \label{eq:F-new-phase-recovery}
\end{equation}
Starting from the final odd endpoint and applying
\eqref{eq:F-new-phase-recovery} to the phases in reverse order recovers the
values at both ends of every phase.

Now suppose the input reservoir and the two endpoints of a phase are known.
Given the current value $X_i$ and a candidate triple $(p,q,e)$, put
\begin{equation}
 W=X_i-H_p-2^eH_q.
 \label{eq:F-new-reverse-block}
\end{equation}
For the correct candidate,
$$
 W=2^{K+u_i+R\mathbf 1_{\{i=1\}}+1}X_{i-1}.
$$
Since $X_{i-1}$ is odd, $\nu_2(W)$ determines $u_i$, and division by the
corresponding power of $2$ determines $X_{i-1}$.  More explicitly, when
$W>0$ a candidate forces
$$
 u_i=\nu_2(W)-K-R\mathbf 1_{\{i=1\}}-1,
 \qquad
 X_{i-1}=\frac{W}{2^{\nu_2(W)}}.
$$
We discard the candidate if $W\leq0$, if the displayed value of $u_i$ does
not lie in $\{0,\ldots,L-1\}$, or if any other condition of a choice record
fails.  Thus a surviving candidate triple determines the whole preceding
block.  By \eqref{eq:F-new-low-bits}, every remaining candidate must satisfy
$$
 X_i\equiv H_p+2^eH_q\pmod {2^K}.
$$
Branching is possible only when two distinct source triples give the same
remainder modulo $2^K$.

\vskip 10pt

\noindent
{\bf Part 6. A decoding strategy for source triples}

Recall that roots are permanent labels inherited from the initial seeds.
All reservoir elements having root $s$ may be viewed as descendants of seed
$Z_s$, or as belonging to the same \emph{seed family}.  If a block has
principal root $r$, secondary root $s$, and exponent $e$, it creates the
directed marker
$$
 r\xrightarrow{e}s.
$$
At the level of roots, this arrow records that an element of seed family $s$
was used as a secondary source in forming an output of seed family $r$, with
multiplier $2^e$.  The output remains in family $r$, since it inherits the
root of its principal source.  The direction is essential: the marker
$r\to s$ is different from $s\to r$.

To make this precise, attach to every reservoir element $H_v$ a vector
$\mathbf a_v\in\mathbb Z^{N_0}$.  This vector records only the dependence of the last $K$ binary digits of $H_v$ on the residues $U_s$. 
For a seed and a block output, respectively,
define
\begin{equation}
 \mathbf a_{Z_s}=\mathbf e_s,
 \qquad
 \mathbf a_v=\mathbf a_p+2^e\mathbf a_q,
 \label{eq:F-new-formal-vectors}
\end{equation}
where $\mathbf e_s$ is the $s$-th standard basis vector and $p,q,e$ are the
sources and exponent used to form $v$.  It follows from
\eqref{eq:F-new-seed-residues} and \eqref{eq:F-new-low-bits} that
\begin{equation}
 H_v\equiv\sum_{s=1}^{N_0}a_v(s)U_s\pmod {2^K}.
 \label{eq:F-new-formal-evaluation}
\end{equation}

Modulo $2^{2L}$, the coefficient vector of an element has a particularly
simple description.  Starting from $v$, repeatedly replace the current
element by the principal source used to create it, until an initial seed is
reached.  Let $\mathcal P(v)$ be the collection of blocks encountered in this
way.  Since roots are inherited from principal sources, all these blocks have
the same principal root
$$
 r=\operatorname{root}(v).
$$
If a block $x\in\mathcal P(v)$ used a secondary source of root $s_x$ and
exponent $e_x$, then it created the directed arrow $r\to s_x$.  We claim that
\begin{equation}
 \mathbf a_v
 \equiv
 \mathbf e_r+
 \sum_{x\in\mathcal P(v)}2^{e_x}\mathbf e_{s_x}
 \pmod {2^{2L}}.
 \label{eq:F-new-ancestry-expansion}
\end{equation}

To see this, first observe from the recursion
\eqref{eq:F-new-formal-vectors} that
$$
 \mathbf a_w\equiv
 \mathbf e_{\operatorname{root}(w)}
 \pmod {2^L}
$$
for every reservoir element $w$.  Indeed, this is true for the seeds, and
the secondary term $2^e\mathbf a_q$ is divisible by $2^L$ because
$e\geq L$.  Consequently, if block $x$ has secondary source $q_x$, then
$$
 2^{e_x}\mathbf a_{q_x}
 \equiv
 2^{e_x}\mathbf e_{s_x}
 \pmod {2^{2L}}.
$$
Thus, modulo $2^{2L}$, a secondary source contributes only the basis vector
of its root.  Every contribution coming from inside the history of that
secondary source contains a second factor of size at least $2^L$ and
therefore vanishes modulo $2^{2L}$.  Iterating along the principal sources
proves \eqref{eq:F-new-ancestry-expansion}.

For a small example, take $K=6$, $L=1$, and seed residues
$$
 (U_1,U_2,U_3)=(3,5,7)\pmod {64}.
$$
A block with principal source $Z_1$ and secondary source $Z_2$ creates the arrow
$1\to2$ and has coefficient vector
$$
 \mathbf e_1+2\mathbf e_2=(1,2,0),
$$
so its last six binary digits equal $3+2\cdot5=13$.  A block with principal
source $Z_1$ and secondary source $Z_3$ creates $1\to3$, has vector $(1,0,2)$, and
has last six binary digits $3+2\cdot7=17$.  The arrow $1\to2$ may not be
used again, although the reverse arrow $2\to1$ remains allowed.

We now record the separation property needed by the decoder.  At the
beginning of a phase, call $(p,q,e)$ \emph{admissible} if $H_p,H_q$ belong to
the input reservoir, their roots are distinct, the arrow
$\operatorname{root}(H_p)\to\operatorname{root}(H_q)$ has not yet been used, and
$L\leq e<2L$.  Set
$$
 \mathbf b(p,q,e)=\mathbf a_p+2^e\mathbf a_q.
$$

\begin{lemma}
\label{lem:F-new-marker-separation}
If $(p,q,e)\neq(p',q',e')$ are two admissible triples at the beginning of the
same phase, then some coordinate $s$ satisfies
\begin{equation}
 \nu_2\bigl(b_s(p,q,e)-b_s(p',q',e')\bigr)<4L.
 \label{eq:F-new-triple-separation}
\end{equation}
Here $\nu_2(x)$ is the $2$-adic valuation of $x$, with
$\nu_2(0)=+\infty$.
\end{lemma}

\begin{proof}
We first note that if $v\neq w$ belong to the same reservoir and have the same
root $r$, then some coordinate distinguishes $\mathbf a_v$ and $\mathbf a_w$
at valuation less than $2L$.  This situation does not arise in
$\mathcal H_0$, which has one seed of every root.  In a later reservoir, let
$r\to s$ be the arrow created when $v$ was born.  The elements $v$ and $w$
were produced in the same preceding phase from old sources.  Hence the arrow
$r\to s$ occurs neither in the principal ancestry of $w$ nor as the arrow
created when $w$ was born: the former ancestry predates that phase, and the
latter arrow is different by the rule forbidding reuse of an ordered pair.
At coordinate $s$,
\eqref{eq:F-new-ancestry-expansion} therefore gives
$$
 a_v(s)-a_w(s)\equiv2^{e_v}\pmod {2^{2L}},
$$
where $L\leq e_v<2L$.  Thus
$$
 \nu_2\bigl(a_v(s)-a_w(s)\bigr)=e_v<2L.
$$
Once created, the arrow $r\to s$ remains used at the beginning of every later
phase.

We now compare the two triples.  If their principal roots differ, their
coefficient vectors already differ modulo $2$.  Suppose their principal roots
agree but $p\neq p'$.  By the preceding paragraph, some previously used arrow
$r\to s$ gives a distinguishing coordinate of valuation less than $2L$.
Admissibility forbids both secondary roots from being $s$.  At coordinate $s$,
$\mathbf a_q$ and $\mathbf a_{q'}$ are therefore divisible by $2^L$, and after
multiplication by $2^e$ and $2^{e'}$, the two new contributions are divisible
by $2^{2L}$ and cannot cancel the old distinction.

It remains to consider $p=p'$.  Suppose first that the secondary roots differ,
and let $s=\operatorname{root}(H_q)$.  At coordinate $s$, the contribution
$2^e\mathbf a_q$ has valuation $e$, whereas
$2^{e'}\mathbf a_{q'}$ is divisible by $2^{e'+L}$.  Since
$e'+L\geq2L>e$, this coordinate distinguishes the two vectors at valuation
$e<2L$.  If the secondary roots agree but $e\neq e'$, the difference at their common root coordinate has valuation $\min(e,e')<2L$.  Finally, if the secondary roots and exponents agree
but $q\neq q'$, the preceding same-root observation, followed by multiplication
by $2^e$, gives a distinguishing coordinate of valuation at most
$$
 (2L-1)+(2L-1)=4L-2.
$$
These cases prove \eqref{eq:F-new-triple-separation}.
\end{proof}

\vskip 10pt

\noindent
{\bf Part 7. Choosing the seeds at random}
 
Let $\mathcal D$ be the set of the $T$ choice records counted above, and let
$$
 \mathcal U=
 \left\{(U_1,\ldots,U_{N_0}):
 1\leq U_s<2^K\text{ and every }U_s\text{ is odd}\right\}.
$$
We use randomness to select one suitable vector in $\mathcal U$.
Equip $\mathcal U$ with the uniform measure, that is, the $U_s$ are
independent and uniform among the $2^{K-1}$ odd residue classes modulo $2^K$.
First fix a record $d\in\mathcal D$.  Its arrows, coefficient vectors, and
candidate triples are then deterministic, and probability refers only to the
choice of vector $\mathbf U\in\mathcal U$.

At the beginning of a phase based on a reservoir of size $N$, let
$\mathcal W$ be the set of all triples admissible for the arrow set at that
time.  We keep this set fixed while decoding the phase, even though some of
its triples cease to be allowed after new arrows are added.  This only enlarges
the candidate set, and
$$
 |\mathcal W|\leq N^2L\leq r^2L.
$$

Consider two distinct candidates.  By Lemma~\ref{lem:F-new-marker-separation}
and \eqref{eq:F-new-formal-evaluation},
some coordinate of the difference of their coefficient vectors is $2^v$
times an odd integer, with $v<4L<K$.  Fix every seed residue except the residue
$U_s$ in that coordinate.  Equality of the two numerical remainders modulo
$2^K$ imposes one linear congruence on $U_s$, with at most $2^v$ solutions
modulo $2^K$.  Therefore the probability of this equality is at most
$$
 \frac{2^v}{2^{K-1}}\leq2^{4L-K}.
$$

For the fixed record $d$, call a block \emph{ambiguous} for $\mathbf U$ if
its chosen triple has the same remainder modulo $2^K$ as another triple in the
corresponding set $\mathcal W$.  The union bound shows that a fixed block is
ambiguous with probability at most
\begin{equation}
\begin{aligned}
\delta_r:=r^2L\,2^{4L-K}
&\le Lr^{-(\rho-1)/4}\\
&=O_c\left((\log r)r^{-(\rho_0-1)/4}\right)=o(1).
\end{aligned}
\label{eq:F-new-delta}
\end{equation}
Here we used \eqref{eq:F-new-parameters}. To obtain the
uniform estimate, observe that $$L\le(\rho-1)\log_2r/16$$
and that the function $x\mapsto x2^{-x/4}$ is decreasing
for all sufficiently large $x$. Since $\rho\ge\rho_0>1$,
this gives
$$
Lr^{-(\rho-1)/4}
\le\frac{(\rho_0-1)\log_2r}{16}
 r^{-(\rho_0-1)/4}
$$
for all sufficiently large $m$.  Let
$A_{\rm amb}(d,\mathbf U)$ be the number of ambiguous blocks.  Since there
are at most $r$ blocks,
$$
 \mathbb E_{\mathbf U}A_{\rm amb}(d,\mathbf U)\leq r\delta_r.
$$
Set $\varepsilon_r=\sqrt{\delta_r}$.  Markov's inequality gives, uniformly in
$d$,
$$
 \mathbb P_{\mathbf U}
 \left(A_{\rm amb}(d,\mathbf U)>\varepsilon_rr\right)
 \leq\varepsilon_r.
$$
Moreover, $\log(r^2L)=O(\log r)$ uniformly, so
\eqref{eq:F-new-delta} gives
\begin{equation}
\varepsilon_r\log(r^2L)
=O_c\left(r^{-(\rho_0-1)/8}(\log r)^{3/2}\right)=o(1).
\label{eq:F-new-epsilon}
\end{equation}
Averaging the Markov estimate over $d\in\mathcal D$ yields
$$
 \mathbb E_{\mathbf U}
 \left|\{d\in\mathcal D:
 A_{\rm amb}(d,\mathbf U)>\varepsilon_rr\}\right|
 \leq\varepsilon_rT.
$$
Consequently, there exists one vector $\mathbf U^*\in\mathcal U$ for which
the set
$$
 \mathcal D_{\rm good}
 =\{d\in\mathcal D:A_{\rm amb}(d,\mathbf U^*)\leq\varepsilon_rr\}
$$
satisfies
\begin{equation}
 |\mathcal D_{\rm good}|\geq(1-\varepsilon_r)T.
 \label{eq:F-new-good-records}
\end{equation}
Fix this seed vector and the corresponding initial reservoir for the rest of
the proof.

\vskip 10pt

\noindent
{\bf Part 8. Completing the proof}

Given a final odd endpoint, we first recover all phase endpoints using
\eqref{eq:F-new-phase-recovery}.  The initial reservoir is known from
$\mathbf U^*$ and we decode the first growth phase backward using
\eqref{eq:F-new-reverse-block}. This recovers its outputs and hence the next
reservoir. We repeat the process through the growth phases and finally
through the terminal phase.

At a nonambiguous block, the remainder in \eqref{eq:F-new-low-bits} determines
the source triple uniquely among the corresponding enlarged set $\mathcal W$.
At an ambiguous block, there are at most $r^2L$ candidates, and each candidate
determines $u_i$ and the preceding endpoint.  We can organize all possible backward reconstructions in a rooted tree whose root is
the given final odd endpoint.  In this tree, a node represents a partial reconstruction,
and its children correspond to the candidate triples that produce a valid
preceding block.  A leaf therefore represents a complete choice record
producing the prescribed endpoint. 
 We retain only branches that end at records in $\mathcal D_{\rm good}$.  If a
node on the path associated with a good record has at least two surviving
children, then the chosen triple of that record has the same residue modulo
$2^K$ as another triple in its phase-start set $\mathcal W$.  The
corresponding block is therefore ambiguous for that record.

Consequently, every good root-to-leaf path contains at most
$\lfloor \varepsilon_rr\rfloor$ branching nodes.  Each branching node has at most $r^2L$
children, while steps with a unique continuation do not increase the number
of leaves.  It follows that every fixed final odd endpoint is produced by at
most
\begin{equation}
 (r^2L)^{\varepsilon_rr}=e^{o(r)}
 \label{eq:F-new-multiplicity}
\end{equation}
good records, by \eqref{eq:F-new-epsilon}.

It follows from \eqref{eq:F-new-T}, \eqref{eq:F-new-good-records}, and
\eqref{eq:F-new-multiplicity} that the number of distinct final odd endpoints
is at least
$$
 \frac{(1-\varepsilon_r)T}{\exp(o(r))}
 \geq m^re^{-O_c(r)}.
$$
Finally, the injective normalization \eqref{eq:F-new-normalization} produces
the same number of integers in $[2^m,2^{m+1})$, each with an addition chain of
length at most $m+r$.  Hence, for some $C>0$ and all sufficiently large
$m$,
$$
 F(m,r)\geq C^rm^r,
$$
which proves the required lower bound in Theorem~\ref{t2}.

\medskip
\noindent
\textbf{Remark.}
A \emph{Brauer chain} is an addition chain in which
each step has the form $a_j=a_{j-1}+a_t$ for some $t<j$.
Let $\ell_*(n)$ be the minimum length of such a chain and put
$$
 F_*(m,r)=\#\{n\in[2^m,2^{m+1}):\ell_*(n)\leq m+r\}.
$$
Every chain produced by the preceding construction is a Brauer chain. Therefore, the same
lower bound holds for $F_*(m,r)$. Since $F_*(m,r)\leq F(m,r)$, Theorem~\ref{t2}
also gives the corresponding upper bound.

\section{Proving the lower bound in  Theorem \ref{t1}}

The proof of the lower bound in Theorem \ref{t1} is also  constructive. This proof will be simpler than that of the lower bound in Theorem \ref{t2}. We construct addition chains (which will not all yield distinct integers) in the following manner.
\begin{enumerate}
\item The addition chains will contain $m$ doubling steps and $r$ additive steps.  First, we choose the locations of the additive steps in the chain. We begin the chain by a sequence of $\lfloor m/2\rfloor$ doubling steps. We choose the locations of the additive steps among the $\lceil m/2+r\rceil$ other steps. The number of ways to make this choice is
$$
{\lceil m/2\rceil+r \choose r}\ge K^r\left(\frac{m}{r}\right)^r
$$
for some positive constant $K$.
\item For each additive step, we choose the added element among the first $\lfloor m/4\rfloor$ elements of the chain.  The number of ways to choose the added elements is thus
$$
\lfloor m/4\rfloor ^r.
$$
\end{enumerate}

From steps 1 and 2 above, it follows that the number of ways of constructing addition chains in such a way is thus at least
$$
(K/8)^r \left(\frac{m^2}{r}\right)^r.
$$
It is also clear that each sequence of choices will lead to a distinct addition chain. Furthermore, since we have $m$ doubling steps, the resulting integer is larger  than $2^m$.  It remains to verify that the resulting integer will be smaller than $2^{m+1}$. 
By construction, each additive step adds one of
$1,2,2^2,\ldots,2^{\lfloor m/4\rfloor-1}$
to a current term at least $2^{\lfloor m/2\rfloor}$.
Thus each additive step increases the current term by a factor
at most
$$
1+2^{\lfloor m/4\rfloor-1-\lfloor m/2\rfloor}
\le 1+2^{-m/4}.
$$
It follows that, in total, the additive steps will make the chain grow by a factor which is less than or equal to
$$
(1+2^{-m/4})^r.
$$
Since $r\le m$,  this factor is less than 2 for $m$ sufficiently large, which guarantees that the endpoint of the chain will be smaller than $2^{m+1}$. This completes the proof of the lower bound in Theorem \ref{t1}.

\vfill
\noindent

\end{document}